\documentclass[10pt]{article}
\usepackage{amsthm,amsmath,amssymb}

\usepackage{tikz,graphicx,url}

\usepackage[colorlinks=true,citecolor=blue,linkcolor=blue,urlcolor=blue]{hyperref}
\usetikzlibrary{matrix}
\newcommand{\arxiv}[1]{\href{http://arxiv.org/abs/#1}{\texttt{arXiv:#1}}}

\theoremstyle{plain}
\newtheorem{theorem}{Theorem}[section]
\newtheorem{lemma}[theorem]{Lemma}
\newtheorem{corollary}[theorem]{Corollary}
\newtheorem{proposition}[theorem]{Proposition}

\theoremstyle{definition}
\newtheorem{definition}[theorem]{Definition}
\newtheorem{example}[theorem]{Example}

\theoremstyle{remark}

\newbox\mybox
\def\centerfigure#1{%
    \setbox\mybox\hbox{#1}%
    \raisebox{-0.5\dimexpr\ht\mybox+\dp\mybox}{\copy\mybox}%
}

\title{\bf The Redei-Berge Hopf algebra of digraphs}

\author{Vladimir Gruji\'c\\
\small Faculty of Mathematics\\[-0.8ex]
\small University of Belgrade\\[-0.8ex]
\small Serbia\\
\small\tt vladimir.grujic@matf.bg.ac.rs\\
\and
Tanja Stojadinovi\'c\\
\small Faculty of Mathematics\\[-0.8ex]
\small University of Belgrade\\[-0.8ex]
\small Serbia\\
\small\tt tanja.stojadinovic@matf.bg.ac.rs}

\begin{document}

\maketitle

\begin{abstract}
In a series of recent talks Richard Stanley introduced a symmetric function associated to digraphs called the Redei-Berge symmetric function. This symmetric function enumerates descent sets of permutations corresponding to digraphs. We show that such constructed symmetric function arises from a suitable structure of combinatorial Hopf algebra on digraphs. The induced Redei-Berge polynomial satisfies the deletion-contraction property which makes it similar to the chromatic polynomial. The Berge's classical result on the number of Hamiltonian paths in digraphs is a consequence of the reciprocity formula for the Redei-Berge polynomial.


\end{abstract}

\section{Introduction}

A {\it digraph} $X$ is a pair $X=(V,E)$, where $V$ is a finite set and $E$ is a collection
$E\subset\{(u,v)\in V\times V | u\neq v\}$. Elements $u\in V$ are vertices and
elements $(u,v)\in E$ are directed edges of the digraph $X$. A $V$-{\it listing} is a list of all vertices with no repetitions, i.e. a bijective map $\sigma:[n]\rightarrow V$. We write $\Sigma_V$ for the set of all $V$-listings. For a $V$-listing $\sigma=(\sigma_1,\ldots,\sigma_n)\in\Sigma_V$, define the $X$-{\it descent set} as

\[X\mathrm{Des}(\sigma)=\{1\leq i\leq n-1 | (\sigma_i,\sigma_{i+1})\in E\}.\] If $V$ is the set $[n]=\{1,\ldots,n\}$ and $E=\{(i,j) | 1\leq j<i\leq n\}$ then $X$-descent sets are standard descent sets of permutations $\sigma\in\Sigma_{[n]}$.

Grinberg and Stanley associated to a digraph $X$ a generating function for $X$-descent sets and named it the Redei-Berge symmetric function, see \cite{GS}, \cite{RS}

\begin{equation}\label{descents}
U_X=\sum_{\sigma\in\Sigma_V}F_{X\mathrm{Des}(\sigma)}.
\end{equation}
It is expanded in the basis of fundamental quasisymmetric functions
\begin{equation}\label{fundamental}
F_I=\sum_{\substack{1\leq i_1\leq i_2\leq\cdots\leq i_n\\
                  i_j<i_{j+1} \ \mathrm{for \ each} \ j\in I}}x_{i_1}x_{i_2}\cdots x_{i_n}, I\subset[n-1].
\end{equation}

The function $U_X$ first appeared in Chow's paper \cite{C}, and later in Wiseman's \cite{W} in 2007, in connection with the path-cycle symmetric function of a digraph.
The symmetic function $U_X$ is related to two classical results in graph theory. The first of them originates in 1933 paper of the famous Hungarian mathematician Laszlo Redei \cite{LR} and states that every finite tournament contains an odd number of Hamiltonian paths. The other one was found by Claude Berge \cite{CB} and it claims that the numbers of Hamiltonian paths of a digraph and its complement are of the same parity.

We introduce the combinatorial Hopf algebra on digraphs $(\mathcal{D},\zeta)$ and obtain the expansion of the enumerator $U_X$ in the basis of monomial quasisymmetric functions. The character $\zeta$ enumerates Hamiltonian paths in the complementary digraph $\overline {X}$. The principal specialization \[u_X(m)=\mathrm{ps}^1(U_X)(m)\] of the symmetric function $U_X$ is a polynomial in $m$ which enumerates $(f,X)$-friendly $V$-listings for colorings of vertices with at most $m$ colors. We show that this polynomial satisfies the deletion-contraction property
\[u_X(m)=u_{X\setminus e}(m)-u_{X/e}(m).\]
We interpreted the Berge result as the reciprocity formula for the polynomial $u_X(m)$.
%It would be GREAT if we could interpret the number of Hamiltonian paths in  $X$ as the value of some character related to $\zeta$.

\section{Basics from combinatorial Hopf algebras}

The theory of combinatorial Hopf algebras is founded in \cite{ABS}.
We review some basic notions and facts. A {\it partition} $\{V_1,\ldots,V_k\}\vdash V$ of the length $k$
of a finite set $V$ is a set of disjoint nonempty subsets with $V_1\cup\ldots\cup V_k=V$. A {\it composition} $(V_1,\ldots,V_k)\models V$ is an ordered partition. Let $n=|V|$ be the number of elements of $V$. A composition $\alpha\models n$ is a sequence $\alpha=(a_1,\ldots,a_k)$ of
positive integers with $a_1+\cdots+a_k=n$. The type of a composition $(V_1,\ldots,V_k)\models V$ is the composition $\mathrm{type}(V_1,\ldots,V_k)=(|V_1|,\ldots,|V_k|)\models n$. A partition $\lambda\vdash n$ is a
composition $(a_1,\ldots,a_k)\models n$ such that $a_1\geq a_2\geq\cdots\geq
a_k$. There is a bijection between the set of compositions $\mathrm{Comp}(n)$ and the power set $2^{[n-1]}$ given by
$(a_1,\ldots,a_k)\mapsto\{a_1,a_1+a_2,\ldots, a_1+\cdots+a_{k-1}\}$. We denote the inverse of this bijection by $I\mapsto\mathrm{comp}(I)$.

A {\it combinatorial Hopf algebra} $(\mathcal{H},\zeta)$ (CHA
for short) over a field $\mathbf{k}$ is a graded, connected Hopf
algebra $\mathcal{H}=\oplus_{n\geq 0}\mathcal{H}_n$ over $\mathbf{k}$ together with
a multiplicative functional
$\zeta:\mathcal{H}\rightarrow\mathbf{k}$ called the {\it character}.

The terminal object in the category of CHA's is the
CHA of quasisymmetric functions $(QSym,\zeta_Q)$. For basics of
quasisymmetric functions see \cite{EC} and for the Hopf algebra structure on $QSym$ see \cite{GR}. A composition $\alpha=(a_1,\ldots,a_k)\models n$ defines the monomial quasisymmetric function

\[M_\alpha=\sum_{i_1<\cdots<
i_k}x_{i_1}^{a_1}\cdots x_{i_k}^{a_k}.\] Alternatively we write $M_I=M_{\mathrm{comp}(I)},\ I\subset[n-1]$. Another basis consists of fundamental quasisymmetric functions $(\ref{fundamental})$ which are expressed in the monomial basis by
\begin{equation}\label{montofund}
F_I=\sum_{I\subset J}M_J,\ I\subset[n-1].
\end{equation}

The antipode $S$ on $QSym$ is determined by its values on fundamental quasisymmetric functions
% \[S(M_I)=(-1)^{|I|+1}\sum_{J\subset I^{\text{op}}}M_J, I\subset[n-1],\]
\begin{equation}\label{fundantipod}
S(F_I)=(-1)^{n}F_{(I^{\text{op}})^c}, I\subset[n-1],
\end{equation}
where $I^{\text{op}}=\{n-i| i\in I\}$ and $(I^{\text{op}})^c$ is its complement in $[n-1]$.

The principal specialization $\mathrm{ps}^1:QSym\rightarrow\mathbf{k}[m]$ is an algebra homomorphism to the polynomial algebra defined by
\[\mathrm{ps}^1(\Phi)(m)=\Phi(\underbrace{1,\ldots,1}_{\text{$m$ ones}},0,0,\ldots).\] On the monomial basis it takes values
\begin{equation}\label{psmonom}
\mathrm{ps}^1(M_I)(m)={m \choose |I|+1}, I\subset[n-1], n>0
\end{equation} and determines the character $\zeta_Q(M_I)=\mathrm{ps}^1(M_I)(1)=\left\{\begin{array}{cc} 1,& I=\emptyset,\\
0,& \text{otherwise}\end{array}\right.$. The reciprocity formula expresses values of the principal specialization at negative integers
\begin{equation}\label{reciprocity}
\text{ps}^1(\Phi)(-m)=\text{ps}^1(S(\Phi))(m).
\end{equation}

The unique canonical morphism
$\Psi:(\mathcal{H},\zeta)\rightarrow(QSym,\zeta_Q)$ is given on
homogeneous elements with

\begin{equation}\label{canonical}
\Psi(h)=\sum_{\alpha\models n}\zeta_\alpha(h)M_\alpha, \
h\in\mathcal{H}_n,
\end{equation} where $\zeta_\alpha$ is the convolution product
\begin{equation}\label{generalcoeff}
\zeta_\alpha=\zeta_{a_1}\cdots\zeta_{a_k}:\mathcal{H}
\stackrel{\Delta^{(k-1)}}\longrightarrow\mathcal{H}^{\otimes
k}\stackrel{proj}\longrightarrow\mathcal{H}_{a_1}\otimes\cdots\otimes\mathcal{H}_{a_k}
\stackrel{\zeta^{\otimes k}}\longrightarrow\mathbf{k}.
\end{equation}

The algebra of symmetric functions $Sym$ is a subalgebra of
$QSym$ and it is the terminal object in the category of
cocommutative combinatorial Hopf algebras.

\section{Hopf algebra of digraphs}

We say that two digraphs $X=(V,E)$ and $Y=(V',E')$ are
isomorphic if there is a bijection of vertices $f:V\rightarrow V'$
such that $(u,v)\in E$ if and only if $(f(u),f(v))\in E'$.
%Any isomorphism class of digraphs $[X]$ has the canonical representative on the vertex set $[n]=\{1<2<\cdots<n\}$ for some integer $n\geq 0$.

The restriction of a digraph $X=(V,E)$ on a subposet $S\subset V$ is the
digraph $X|_S=(S,E|_S)$, where $E|_S=\{(u,v)\in E | u,v\in S\}$. For digraphs $X=(V,E)$ and $Y=(V',E')$ we define the product $X\cdot Y$ as the digraph on the disjoint union $V\sqcup V'$ with the set of directed edges \[E\cup E'\cup\{(u,v) | u\in V, v\in V'\}.\] The product of digraphs is obviously an associative, but not a commutative operation.

Let $\mathcal{D}=\oplus_{n\geq0}\mathcal{D}_n$ be the graded vector space over the field of rational numbers $\mathbb{Q}$, which is linearly spanned by the set of all isomorphism classes of digraphs, where the grading is given by the number of vertices. The linear extension of the product on digraphs determines the multiplication $\mu:\mathcal{D}\otimes\mathcal{D}\rightarrow\mathcal{D}$, which turns the space $\mathcal{D}$ into a noncommutative algebra. The restrictions of digraphs may be used to define a comultiplication $\Delta:\mathcal{D}\rightarrow\mathcal{D}\otimes\mathcal{D}$ by

\[\Delta([X])=\sum_{S\subset V}[X|_S]\otimes[X|_{V\setminus S}],\] where $V$ is the vertex set of a digraph $X$. Evidently, this is a coassociative and a cocommutative operation. It is easy to check that $\Delta$ is an algebra morphism, meaning that

\[\Delta([X\cdot Y])=\Delta([X])\cdot\Delta([Y])\] for any isomorphism classes $[X]$ and $[Y]$ of digraphs.

\begin{example}
Some examples of multiplication and comultiplication in $\mathcal{D}$ are the following:
\begin{center}
$[1 \ 2]\cdot[1\leftarrow 2]=$\Big[\centerfigure{\small{\begin{tikzpicture}\small{\node (3) at (0, 0) {3};
\node (1) at (0, 1) {1};
\node (2) at (1, 1) {2};
\node (4) at (1,0) {4};
\draw[->]
  (1) edge (3) (1) edge (4) (2) edge (3) (2) edge (4) (4) edge (3);}
\end{tikzpicture}}}\Big],
\end{center}

\[\Delta\Big(\Big[\centerfigure{\small{\begin{tikzpicture}
\node (1) at (0, 0) {1};
\node (2) at (-0.5, -1) {2};
\node (3) at (0.5,-1) {3};
\draw[->]
  (1) edge (3) (1) edge (2);
\end{tikzpicture}}}\Big]\Big)=\Big[\centerfigure{\small{\begin{tikzpicture}
\node (1) at (0, 0) {1};
\node (2) at (-0.5, -1) {2};
\node (3) at (0.5,-1) {3};
\draw[->]
  (1) edge (3) (1) edge (2);\end{tikzpicture}}}\Big]\otimes[\emptyset]+2[1\rightarrow 2]\otimes[1]+[1 \ 2]\otimes[1]+\]\[+[1]\otimes[1 \ 2]+2[1]\otimes[1\rightarrow 2]+[\emptyset]\otimes\Big[\centerfigure{\small{\begin{tikzpicture}
\node (1) at (0, 0) {1};
\node (2) at (-0.5, -1) {2};
\node (3) at (0.5,-1) {3};
\draw[->]
  (1) edge (3) (1) edge (2);\end{tikzpicture}}}\Big].\]
\end{example}

With these operations the space of digraphs $\mathcal{D}$ is endowed with a structure of a graded, connected, noncommutative and cocommutative Hopf algebra. The unit element $\iota$ is given by the class
$[\emptyset]$ of the digraph on the empty set of vertices and the counit is given by $\epsilon([\emptyset])=1$ and $\epsilon([X])=0$ otherwise.

The antipode of
$\mathcal{D}$ is determined by the Takeuchi formula for the antipode of a graded bialgebra with $S([\emptyset])=[\emptyset]$ and

\[S([X])=\sum_{k\geq 1}(-1)^k\sum_{(V_1,\ldots,V_k)\models V}[X|_{V_1}\cdots X|_{V_k}], X\neq\emptyset,\]
where the inner sum goes over all compositions of the length $k$ of the vertex
set $V$ of a digraph $X$.

We will define some natural operations on digraphs and observe their obvious properties.

\begin{definition}
For a digraph $X=(V,E)$, let $\overline{X}=(V, E^c)$ and $X^\text{op}=(V, E^\text{op})$ be its {\it complementary} and {\it opposite digraph} respectively, whose directed edges are determined by
\[(u,v)\in E^c \ \text{if and only if} \ (u,v)\notin E,\]
\[(u,v)\in E^\text{op} \ \text{if and only if} \ (v,u)\in E.\]
\end{definition}
\begin{lemma}\label{conjopp}
The operations of taking the complementary and the opposite digraphs satisfy
\[\overline{X\cdot Y}=\overline{Y}\cdot\overline{X}, (X\cdot Y)^\text{op}=Y^\text{op}\cdot X^\text{op},\]
\[(X^\text{op})^\text{op}=X, \overline{(\overline{X})}=X,\]
\[\overline{X^\text{op}}=(\overline{X})^\text{op}\] and consequently descend to commuting involutional anti-isomorphisms of the algebra $\mathcal{D}$.
\end{lemma}

Define reversion of a $V$-listing $\sigma=(\sigma_1,\ldots,\sigma_n)\in\Sigma_V$ by $\mathrm{rev}\sigma=(\sigma_n,\ldots,\sigma_1)$. Recall that for a subset $I\subset[n-1]$ we denoted by $I^{op}=\{n-i|i\in I\}$ and $I^c=[n-1]\setminus I$ its opposite and complementary subset.

\begin{lemma}\label{descentopcompl}
For a digraph $X$ on the vertex set $V$ and a $V$-listing $\sigma\in\Sigma_V$ it holds
\[X^\text{op}\mathrm{Des}(\mathrm{rev}\sigma)=(X\mathrm{Des}(\sigma))^\text{op},\]
\[\overline{X}\mathrm{Des}(\sigma)=(X\mathrm{Des}(\sigma))^c.\]
\end{lemma}

To obtain a structure of a combinatorial Hopf algebra on digraphs we need a character.
\begin{definition}
A $V$-listing $\sigma=(\sigma_1,\ldots,\sigma_n)\in\Sigma_V$ is a {\it Hamiltonian path} of a digraph $X$ if $X\mathrm{Des}(\sigma)=[n-1]$. This means that $(\sigma_j,\sigma_{j+1})\in E$ for all $j=1,\ldots,n-1$.
\end{definition}
Let $\zeta:\mathcal{D}\rightarrow\mathbb{Q}$ be the linear functional that counts Hamiltonian paths of the complementary digraph (which is an invariant of  isomorphism classes of digraphs)

\begin{equation}\label{character}
\zeta([X])=\#\{\sigma\in\Sigma_V | X\mathrm{Des}(\sigma)=\emptyset\}.
\end{equation}

\begin{proposition}
The linear functional $\zeta$ is multiplicative on the Hopf algebra of digraphs $\mathcal{D}$.
\end{proposition}
\begin{proof}
Let $X$ and $Y$ be digraphs. Each Hamiltonian path of $\overline{X\cdot Y}=\overline{Y}\cdot\overline{X}$ is uniquely broken into Hamiltonian paths of $\overline{Y}$ and $\overline{X}$. On the other hand any pair of Hamiltonian paths of $\overline{Y}$ and $\overline{X}$ produces by concatenation a Hamiltonian path of
$\overline{X\cdot Y}$.
\end{proof}

We have constructed a combinatorial Hopf algebra of digraphs $(\mathcal{D},\zeta)$. The universal morphism $\Psi:\mathcal{D}\rightarrow QSym$ assigns to each digraph $X$ a quasisymmetric function $\Psi([X])$ determined by $(\ref{canonical})$. If $X$ is a digraph on the vertex set $V$ the coefficients of this quasisymmetric function in the monomial basis are determined by $(\ref{generalcoeff})$ with
\begin{equation}\label{zeta}
\zeta_\alpha([X])=\sum_{\substack{(V_1,\ldots,V_k)\models V\\
                  \mathrm{type}(V_1,\ldots,V_k)=\alpha}}\zeta([X|_{V_1}])\cdots\zeta([X|_{V_k}]).
\end{equation}

Since $\mathcal{D}$ is a cocommutative combinatorial Hopf algebra we actually have that $\Psi([X])$ is a symmetric function, that is, the values of coefficients $\zeta_\alpha$ depend only on partitions of the vertex set $V$.

\section{A generating function for X-descent sets}

From the definition $(\ref{descents})$ of the Redei-Berge symmetric function we have
\begin{equation}\label{fundamentalexp}
U_X=\sum_{\sigma\in\Sigma_V}F_{X\mathrm{Des}(\sigma)}=\sum_{\sigma\in\Sigma_V}\sum_{X\mathrm{Des}(\sigma)=I}F_I=\sum_{I\subset[n-1]}\lambda_I(X)F_I,
\end{equation}
where the coefficients are given by \[\lambda_I(X)=\#\{\sigma\in\Sigma_V | X\mathrm{Des}(\sigma)=I\}.\]
In the basis of monomial quasisymmetric functions we obtain using $(\ref{montofund})$
\begin{equation}\label{monomial}
U_X=\sum_{\sigma\in\Sigma_V}\sum_{X\mathrm{Des}(\sigma)\subset I}M_I=\sum_{I\subset[n-1]}\mu_I(X)M_I,
\end{equation} where
\[\mu_I(X)=\#\{\sigma\in\Sigma_V | X\mathrm{Des}(\sigma)\subset I\}.\]

Grinberg and Stanley in \cite[Theorem 1.31]{GS} expanded the function $U_X$ in the power sum basis of the algebra of symmetric functions. As a part of this expansion they obtained an interpretation of $U_X$ as a certain enumerator of colorings of vertices with positive integers $\mathbb{P}=\{1,2,3,\ldots\}$. We emphasize this interpretation in the following proposition.

\begin{definition}
Let $X$ be a digraph on the vertex set $V$. For a coloring of vertices with positive integers $f:V\rightarrow\mathbb{P}$, a $V$-listing $\sigma=(\sigma_1,\ldots,\sigma_n)\in\Sigma_V$ is called $(f,X)$-{\it friendly} if \[f(\sigma_1)\leq f(\sigma_2)\leq\cdots\leq f(\sigma_n) \ \text{and}\]  \[f(\sigma_j)<f(\sigma_{j+1}) \ \mathrm{for \ each} \ j\in[n-1] \ \mathrm{satisfying} \ (\sigma_j,\sigma_{j+1})\in X.\]
\end{definition}

Denote by $\Sigma_V(f,X)$ the set of all $(f,X)$-friendly $V$-listings and by $\delta_f:\Sigma_V\rightarrow\{0,1\}$ its indicator function. For a coloring $f:V\rightarrow\mathbb{P}$ we write $\mathbf{x}_f=\prod_{v\in V}x_{f(v)}$.

\begin{proposition}[{\cite[section 2.11]{GS}}]\label{nature}
The Redei-Berge symmetric function $U_X$ of a digraph $X$ on the vertex set $V$ is given by
\[U_X=\sum_{f:V\rightarrow\mathbb{P}}\sum_{\sigma\in\Sigma_V}\delta_f(\sigma)\mathbf{x}_f.\]
\end{proposition}
\begin{proof}
The fundamental quasisymmetric function $F_{X\mathrm{Des}(\sigma)}$ defined by $(\ref{fundamental})$ can be described as
\[F_{X\mathrm{Des}(\sigma)}=\sum_{f:V\rightarrow\mathbb{P}}\delta_f(\sigma)\mathbf{x}_f.\] Summing over all $\sigma\in\Sigma_V$ gives
\[U_X=\sum_{\sigma\in\Sigma_V}\sum_{f:V\rightarrow\mathbb{P}}\delta_f(\sigma)\mathbf{x}_f=
\sum_{f:V\rightarrow\mathbb{P}}\sum_{\sigma\in\Sigma_V}\delta_f(\sigma)\mathbf{x}_f.\]
\end{proof}

By Proposition \ref{nature} the symmetric function $U_X$ is described as an integer points enumerator. This was immanent in the theory of generalized permutohedra, see \cite{GPS} for details. Let $\mathcal{F}=(V_1,\ldots,V_k)\models V$ be a composition of the set of vertices. It is of the length $|\mathcal{F}|=k$ and of the type $\mathrm{type}(\mathcal{F})=((|V_1|,\ldots,|V_k|)$. Denote by $\sigma_\mathcal{F}^\circ$ the set of all colorings $f:V\rightarrow\mathbb{P}$ that satisfy
\begin{itemize}
\item $f(v_p)=f(v_q)$ \ \text{whenever} \ $v_p$ \ \text{and} \ $v_q$ \ \text{lies in the same block of} \ $\mathcal{F}$;
\item $f(v_p)<f(v_q)$ \ \text{whenever} \ $v_p$ \ \text{lies in an earlier block of} \ $\mathcal{F}$ \ \text{than} \ $v_q$.
\end{itemize}
We say that a composition $\mathcal{F}\models V$ is the compositional type of a coloring $f\in\mathbb{P}^V$, that is, $f$ is an $\mathcal{F}$-coloring, if and only if $f\in\sigma_{\mathcal{F}}^\circ$. Each coloring $f$ has a unique compositional type. Let $M_\mathcal{F}$ be the enumerator of $\mathcal{F}$-colorings
\[M_\mathcal{F}=\sum_{f\in c_\mathcal{F}^\circ}\mathbf{x}_f.\] It depends only on the type of a composition, i.e. $M_\mathcal{F}=M_{\mathrm{type}(\mathcal{F})}$.
The following lemma claims that the indicator $\delta_f$ depends only on the compositional type of a coloring $f$.

\begin{lemma}
If $f,g\in\mathbb{P}^V$ are colorings of vertices of a digraph $X$ with the same compositional type then $\Sigma_V(f,X)=\Sigma_V(g,X)$.
\end{lemma}
\begin{proof} It follows by observing that for any two vertices $u,v\in V$, the inequality $f(u)<f(u)$ is equivalent to $g(u)<g(v)$ (and likewise for equalities and weak inequalities).
\end{proof}
According to the previous lemma the following definition is correct.

\begin{definition}
For a digraph $X$ a $V$-listing $\sigma\in\Sigma_V$ is called $(\mathcal{F},X)$-{\it friendly} if it is $(f,X)$-friendly for some coloring $f\in c_\mathcal{F}^\circ$. We write $\Sigma_V(\mathcal{F},X)$ for the set of all $(\mathcal{F},X)$-friendly $V$-listings of $X$.
\end{definition}

We give another characterization for $(\mathcal{F},X)$-friendly $V$-listings, which leads to their total number.
\begin{lemma}\label{friendly}
Let $\mathcal{F}=(V_1,\ldots,V_k)\models V$ be a composition and $X_i=X|_{V_i}, i=1,\ldots,k$ be the restrictions of $X$. A $V$-listing $\sigma\in\Sigma_V$ is $(\mathcal{F},X)$-friendly if and only if it can be written as the concatenation $\sigma=\sigma^1\cdots\sigma^k$ of Hamiltonian paths $\sigma^i\in\Sigma_{V_i}$ of complementary digraphs $\overline{X}_i$. The total number of $(\mathcal{F},X)$-friendly $V$-listings is given by
\[|\Sigma_V(\mathcal{F},X)|=\zeta([X_1])\cdots\zeta([X_k]).\]
\end{lemma}

\begin{proposition}\label{cones}
The symmetric function $U_X$ of a digraph $X$ on the vertex set $V$ has the following expansion
\[U_X=\sum_{\mathcal{F}}|\Sigma_V(\mathcal{F},X)|M_\mathcal{F}.\]
\end{proposition}
\begin{proof}
The identity follows from Proposition \ref{nature}
\[U_X=\sum_{f:V\rightarrow\mathbb{P}}\sum_{\sigma\in\Sigma_V}\delta_f(\sigma)\mathbf{x}_f=\sum_{\mathcal{F}}\sum_{f\in c_\mathcal{F}^\circ}\big(\sum_{\sigma\in\Sigma_V}\delta_f(\sigma)\big)\mathbf{x}_f=\sum_{\mathcal{F}}|\Sigma_V(\mathcal{F},X)|\sum_{f\in c_\mathcal{F}^\circ}\mathbf{x}_f.\]
\end{proof}

\begin{theorem}
Let $\Psi:\mathcal{D}\rightarrow QSym$ be the universal morphism from the combinatorial Hopf algebra of digraphs to quasisymmetric functions. Then for a digraph $X$
\[\Psi([X])=U_X.\]
\end{theorem}
\begin{proof}
According to the expansion of Proposition \ref{cones} we have
\[U_X=\sum_{\alpha\models|V|}\Big(\sum_{\text{type}(\mathcal{F})=\alpha}|\Sigma_V(\mathcal{F},X)|\Big)M_\alpha,\] which shows that
\[\mu_I(X)=\sum_{\text{type}(\mathcal{F})=\text{comp}(I)}|\Sigma_V(\mathcal{F},X)|.\] By the characterization \ref{friendly} we have further
\[\mu_I(X)=\sum_{\text{type}(V_1,\ldots,V_k)=\text{comp}(I)}\zeta([X|_{V_1}])\cdots\zeta([X|_{V_k}]).\] This is exactly the coefficient $(\ref{zeta})$
\[\mu_I(X)=\zeta_{\text{comp}(I)}([X]).\]

\end{proof}

We are interested in relations between symmetric functions $U_X, U_{X^\text{op}}$ and $U_{\overline{X}}$.

\begin{proposition}\label{opposite}
The Redei-Berge symmetric functions of a digraph $X$ and its opposite digraph $X^\text{op}$ are equal:
\[U_{X^\text{op}}=U_X.\]
\end{proposition}
\begin{proof}
For a $V$-listing $\sigma\in\Sigma_V$ of vertices $V$ it follows from \ref{descentopcompl} that $X\mathrm{Des}(\sigma)\subset I$ if and only if $X^\text{op}\mathrm{Des}(\mathrm{rev}\sigma)\subset I^\text{op}$ for any subset $I\subset[n-1]$. Therefore $\mu_I(X)=\mu_{I^\text{op}}(X^\text{op}), I\subset[n-1]$.
The operation of reversion can be defined on the monomial basis by $\mathrm{rev}M_I=M_{I^\text{op}}$ and linearly extended to all quasisymmetric functions.
Using expansion $(\ref{monomial})$ we obtain
\[U_X=\mathrm{rev}U_{X^\text{op}}.\] The statement follows since the reversion on symmetric functions is an identity.
\end{proof}

The operation of taking the complementary digraph leads to the Chow antipode formula \cite[Theorem 8.1]{GS}, which has been proved in \cite{C}. We reprove this formula avoiding the power sum expansion.

\begin{theorem}\label{antipode}
The antipode $S$ acts on the Redei-Berge symmetric function $U_X$ of a digraph $X$ on the vertex set $V$ by
\[S(U_X)=(-1)^{|V|}U_{\overline{X}}.\]
\end{theorem}
\begin{proof}
By the expansion $(\ref{fundamentalexp})$ of $U_X$ and the action of $S$ on fundamental quasisymmetric functions $(\ref{fundantipod})$ we have
\[S(U_X)=\sum_{I}\lambda_I(X)S(F_I)=(-1)^{|V|}\sum_I\lambda_I(X)F_{(I^\text{op})^c}.\] Changing the summing variable $J=(I^\text{op})^c$ gives \[S(U_X)=(-1)^{|V|}\sum_J\lambda_{(J^c)^\text{op}}(X)F_J.\] It follows from Lemma \ref{descentopcompl} that
\[\lambda_{(J^c)^\text{op}}(X)=\lambda_{J^c}(X^\text{op})=\lambda_J(\overline{X^\text{op}})=\lambda_J((\overline{X})^\text{op}).\] The last identity is a consequence of Lemma \ref{conjopp}. Finally, Proposition \ref{opposite} implies $\lambda_J((\overline{X})^\text{op})=\lambda_J(\overline{X})$ which finishes the proof.
\end{proof}

\section{The Redei-Berge polynomial}

We call the principal specialization $u_X(m)=\mathrm{ps}^1(U_X)(m)$ the {\it Redei-Berge polynomial}. It is an immediate consequence of Proposition \ref{nature} that
\[u_X(m)=\sum_{f:V\rightarrow[m]}|\Sigma_V(f,X)|,\] which reveals the combinatorial nature of this polynomial.

\begin{proposition}
The Redei-Berge polynomial $u_X(m)$ counts $(f,X)$-friendly $V$-listings for colorings $f:V\rightarrow[m]$ with at most $m$ colors.
\end{proposition}

According to the expansion $(\ref{monomial})$ and the principal specialization of monomial quasisymmetric functions $(\ref{psmonom})$ we obtain for $X\neq\emptyset$
\[u_X(m)=\sum_{I\subset[n-1]}\mu_I(X){m \choose |I|+1}=\sum_{k=1}^n\Big(\sum_{|I|=k-1}\mu_I(X)\Big){m \choose k}.\] If $\mu_k=\sum_{|I|=k-1}\mu_I(X), k=1,\ldots,n$ we have
\begin{equation}\label{RBpol}
u_X(m)=\mu_1{m \choose 1}+\mu_2{m \choose 2}+\cdots+\mu_n{m \choose n}.
\end{equation} The polynomial $u_X(m)$ is integer-valued by definition. It turns out that it satisfies the {\it deletion-contraction property}, similarly to the chromatic polynomial of the underlying graph.
\begin{definition}
Let $X=(V,E)$ be a digraph. The {\it underlying graph} $\widetilde{X}=(V,\widetilde{E})$ of the digraph $X$ is the graph on $V$ with the set of edges $\widetilde{E}=\{\{u,v\}|(u,v)\in E\}$.
\end{definition}
\begin{definition}
The {\it deletion} of an edge $e\in E$ from a digraph $X=(V,E)$ is the digraph $X\setminus e=(V,E\setminus\{e\})$. The {\it contraction} of $X$ by $e=(u,v)\in E$ is the digraph $X/e=(V', E')$, where $V'=V\setminus\{u,v\}\cup\{e\}$ and $E'$ contains all edges in $E$ with vertices different from $u,v\in V$ and additionally for $w\neq u,v$ we have
\begin{itemize}
\item $(w,e)\in E'$ if and only if $(w,u)\in E$ and
\item $(e,w)\in E'$ if and only if $(v,w)\in E$.
\end{itemize}
\end{definition}
\begin{example}
For the digraph $X$ on the set of vertices $V=\{1,2,3\}$ and the set of edges $E=\{(1,2),(2,1),(1,3),(3,2)\}$ the underlying graph is the complete graph on three vertices $\widetilde{X}=K_3$. The contractions $X/(1,2)$ and $X/(2,1)$ are the empty and complete digraph on two-vertex set, respectively.
\end{example}

\begin{theorem}
The Redei-Berge polynomial of a digraph $X=(V,E)$ satisfies the deletion-contraction property
\[u_X(m)=u_{X\setminus e}(m)-u_{X/e}(m), e\in E.\]
\end{theorem}
\begin{proof}
We observe that $\Sigma_V(f,X)\subset\Sigma_V(f,X\setminus e)$ for a coloring $f:V\rightarrow[m]$. For $e=(u,v)\in E$, the difference $\Sigma_V(f,X\setminus e)\setminus\Sigma_V(f,X)$ contains all $V$-lists $\sigma=(\sigma_1,\ldots,\sigma_n)$ with properties
\[f(\sigma_1)\leq\cdots\leq f(\sigma_n),\]
\[f(\sigma_i)<f(\sigma_{i+1}), (\sigma_i,\sigma_{i+1})\in E\setminus e \ \text{and}\]
\[(u,v)=(\sigma_j,\sigma_{j+1}) \ \text{for \ some} \ j=1,\ldots,n-1 \ \text{and} \ f(u)=f(v).\] Such a list determines the $V'$-list \[\widehat{\sigma}=\left\{\begin{array}{ccc} ((u,v),\sigma_3,\ldots,\sigma_n)& j=1,\\
(\sigma_1,\ldots,\sigma_{j-1},(u,v),\sigma_{j+2},\ldots,\sigma_n)& 1<j<n-1,\\ (\sigma_1,\ldots,\sigma_{n-2},(u,v))& j=n-1\end{array}\right.\] on the set of vertices $V'$ of the digraph $X/e$. Let $\widetilde{f}:V'\rightarrow[m]$ be the coloring induced by $f$ with $\widetilde{f}(w)=\left\{\begin{array}{cc} f(w),& w\neq e,\\
f(u)=f(v),& w=e\end{array}\right.$. Then $\widehat{\sigma}$ is a $(\widetilde{f}, X/e)$-friendly $V'$-list. On the other hand, it is easy to see that any $(\widetilde{f},X/e)$-friendly $V'$-list for some coloring $\widetilde{f}:V'\rightarrow[m]$ is obtained uniquely in this way.
\end{proof}
The deletion-contraction property completely determines the Redei-Berge polynomial $u_X(m)$ by recursion. The difference from the chromatic polynomial of the underlying graph $\chi_{\widetilde{X}}(m)$ occurs as a consequence of different initial values.
\begin{example}
\[u_{[1\rightarrow 2]}(m)=m^2, \chi_{[1 \ \noindent\rule{0,3cm}{0.4pt} \ 2]}(m)=m^2-m.\]
\end{example}
An easy inductive argument shows the following nontrivial fact.
\begin{corollary}
The polynomial $u_X(m)$ is with integer coefficients.
\end{corollary}
%This follows from the power sum expansion of $U_X$ and the fact that coefficients of this expansion are integral \cite[Corollary 1.35] {GS}.
The reciprocity formula $(\ref{reciprocity})$ and the antipode formula of Theorem $\ref{antipode}$ give the following identity.

\begin{theorem}\label{polynomial}
The Redei-Berge polynomial satisfies \[u_X(-m)=(-1)^{|V|}u_{\overline{X}}(m).\]
\end{theorem}
\begin{example}
Let $D_n$ be the empty digraph on $n$ vertices. The complementary digraph $\overline{D}_n$ is complete and $u_{\overline{D}_n}(m)=m(m-1)\cdots(m-n+1)$. The reciprocity formula gives $u_{D_n}(m)=m(m+1)\cdots(m+n-1)$.
\end{example}

%\begin{proposition}
%The Redei-Berge polynomial of a digraph $X=(V,E)$ is determined by
%\[u_X(m)=\sum_{\pi\vdash V}(-1)^{|V|-|\pi|}m(m+1)\cdots(m+|\pi|-1)\mu_X(\pi).\] 
%\end{proposition}

The Berge theorem about Hamiltonian paths of a digraph is a consequence of the reciprocity theorem for the Redei-Berge polynomial $u_X(m)$.

\begin{theorem}
The numbers of Hamiltonian paths of a digraph $X$ and its complementary digraph $\overline{X}$ are of the same parity.
\end{theorem}
\begin{proof}
It follows from the definition $(\ref{character})$ of the character $\zeta$ that \[\mu_\emptyset(\overline{X})=\zeta([{\overline{X}}])\] is the number of Hamiltonian paths of $X$. We have \[u_X(1)=\mathrm{ps}^1(U_X)(1)=\mu_\emptyset(X).\] On the other hand, from Theorem \ref{polynomial} we obtain \[u_X(-1)=(-1)^{|V|}u_{\overline{X}}(1)=(-1)^{|V|}\mu_{\emptyset}(\overline{X}).\] Since $u_X(1)$ and $u_X(-1)$ are of the same parity the statement follows.
\end{proof}

\section{Tournaments}

A {\it tournament} is a digraph $X=(V,E)$ such that whenever $u$ and $v$ are two distinct vertices in $V$, exactly one of $(u,v)$ and $(v,u)$ is an edge of $X$. Tournaments are obviously closed under operations of taking restrictions and products, so their isomorphism classes generate a Hopf subalgebra $\mathcal{T}\subset\mathcal{D}$ of the Hopf algebra of digraphs. We call $\mathcal{T}$ the Hopf algebra of tournaments. The complementary digraph $\overline{X}$ of a tournament $X$ is a tournament itself and is equal to its opposite digraph $\overline{X}=X^\text{op}$. The Proposition \ref{opposite} and the Theorem \ref{antipode} imply

\begin{corollary}
For each tournament we have $U_X=U_{\overline{X}}$. The antipode acts on $U_X$ by $S(U_X)=(-1)^{|V|}U_X$.
\end{corollary}
\begin{corollary}
The Redei-Berge polynomial of a tournament $X$ on $n$ vertices satisfies $u_X(-m)=(-1)^{n}u_X(m)$.
\end{corollary}

%\section*{Funding}

%The first author was supported by the Science Fund of the Republic of Serbia, grant no. 7744592, Integrability and Extremal Problems in Mechanics, Geometry and %Combinatorics-MEGIC.

%\noindent The second author was supported by the Science Fund of the Republic of Serbia, grant no. 7749891, Graphical Languages-GWORDS.


%% References


\bibliographystyle{model1a-num-names}
\bibliography{<your-bib-database>}



%\begin{thebibliography}{00}

\end{document}